\documentclass[12pt, a4paper]{article}

\usepackage{pdfsync}
\usepackage{amsmath}\multlinegap=0pt
\usepackage[latin1]{inputenc}
\usepackage{amsthm}
\usepackage{amssymb}
\usepackage[francais,english]{babel}
\usepackage{hyperref}
\usepackage{graphicx}
\usepackage{cancel} 

\numberwithin{equation}{section}
\theoremstyle{plain}
\newtheorem{thm}{Theorem}[section]

\newtheorem{prop}[thm]{Proposition}
\newtheorem{lem}[thm]{Lemma}

\theoremstyle{definition}

\theoremstyle{remark}
\newtheorem{rem}[thm]{Remark}

\newcommand{\R}{\mathbb{R}}

\newcommand\N{{\mathbb N}}
\newcommand\tx{{\tilde x}}

\newcommand\Z{{\mathbb Z}}
\newcommand\T{{\mathbb T}}

\newcommand\pref[1]{(\ref{#1})}

\let \eps\varepsilon

\newcommand\A{{\cal A}}

\newcommand{\g}{\gamma}

\newcommand{\imgbox}[1]{\setlength{\fboxsep}{0pt}\fbox{#1}}
\newcommand{\Om}{\Omega}

\newcommand{\ddd}{:=}

\DeclareMathOperator{\Hil}{Hil}

\DeclareMathOperator{\id}{id}
\def\H {\operatorname{H}}

\newcommand\Sdif{\mathop{\mathrm{Sdiff}}\nolimits}
\newcommand\Law{\mathop{\mathrm{Law}}\nolimits}

\newcommand\proj{\mathop{\mathrm{proj}}\nolimits}

\newcommand\Ent{\mathop{\mathrm{Ent}}\nolimits}

\newcommand\dist{\mathop{\mathrm{dist}}\nolimits}

\newcommand\dive{\mathrm{div}}

\newcommand\GIF{\mathrm{GIF}}

\def\<#1,#2>{\left<#1,#2\right>}

\newcommand\xx{\mathbf{x}}

\newcommand\Leb{{\cal L}}
\newcommand\Tc{{\cal T}}
\newcommand\E{{\cal E}}

\def\PP{{\cal P}}
\def\D{{\cal D}}

\newcommand{\TabOne}[1]{ % 
\begin{tabular}{|p{12cm}|}
 #1
\end{tabular}
}

\newcommand{\TabFour}[1]{ % 
\begin{tabular}{@{}c@{\hspace{1mm}}c@{\hspace{1mm}}c@{\hspace{1mm}}c@{}}
 #1
\end{tabular}
}
\newcommand{\TabFive}[1]{ % 
\begin{tabular}{@{}c@{\hspace{1mm}}c@{\hspace{1mm}}c@{\hspace{1mm}}c@{\hspace{1mm}}c@{}}
 #1
\end{tabular}
}

\title{Generalized incompressible flows, multi-marginal transport  and Sinkhorn algorithm}

\author{Jean-David Benamou \thanks{INRIA-Paris, MOKAPLAN,  rue Simone Iff, 75012, Paris, France 
\texttt{Jean-David.Benamou@inria.fr.}}  $\,^{\dagger}$
\and
Guillaume Carlier\thanks{Universit\'{e} Paris-Dauphine, PSL Research University, CEREMADE (UMR CNRS 7534), 75016 Paris, France 
\texttt{carlier@ceremade.dauphine.fr}} $\,^{*}$
\and
Luca Nenna\thanks{CNRS and Universit\'{e} Paris-Dauphine, PSL Research University, CEREMADE (UMR CNRS 7534), 75016 Paris, France  \texttt{nenna@ceremade.dauphine.fr}}  %$\,^{**}$
}

\begin{document}

\maketitle

\begin{abstract}
Starting from Brenier's relaxed formulation of the incompressible Euler equation in terms of geodesics in the group of measure-preserving diffeomorphisms, we propose a numerical method based on Sinkhorn's algorithm for the entropic regularization of optimal transport. We also make a detailed comparison of this entropic regularization with the so-called Bredinger  entropic interpolation problem (see \cite{leonardBredinger}). 
Numerical results in dimension one and two illustrate the feasibility of the method. 

\end{abstract}


\medskip

\textbf{MS Classification:} 76M30, 65K10.

\section{Introduction}\label{sec-intro}
The motion of  incompressible inviscid fluids inside a bounded domain $\D\subset\R^d$ (or, as we shall often consider the periodic in space case, $\D=\T^d:=\R^d / 2\pi\Z^d$ is the flat torus) without the action of external forces is governed by the equations introduced by Euler in 1755  \cite{euler1755principes}:

\begin{equation}
\label{euler-eq}
\begin{cases}
\partial_{t}u+(u\cdot\nabla)u+\nabla p=0\quad&\text{in}\;(0,T)\times\D\\
\dive(u)=0\quad&\text{in}\; (0,T)\times\D\\
u\cdot n=0\quad&\text{on}\;(0,T)\times\partial\D, 
\end{cases}
\end{equation}
where $n$ denotes the unit normal to $\partial \D$, $u$ denotes the velocity field and $p$ is the pressure. As was first emphasized by Arnold \cite{arnold66}, also see Arnold and Khesin \cite{arnoldkhesin}, \pref{euler-eq} can be seen, at least formally, in Lagrangian coordinates as the Euler-Lagrange for the minimization of the action
\begin{equation}
\label{arnold} 
\A(X):=\int_0^T \Vert \dot{X}\Vert^2_{L^2(\D)} \mbox{d} t
\end{equation}
subject to the constraint that $t\mapsto X(t,.)$ is  a path in $\Sdif$, the group of Lebesgue-measure preserving diffeomorphisms of $\D$. Indeed, the incompressibility constraint  translates in Eulerian terms as the requirement that the velocity  field $u$ associated with $X$, through $\partial_t X(t,x)=u(t, X(t,x))$ is divergence-free. The pressure $p$ acts as a Lagrange multiplier for this constraint  and the optimality equation for the minimization of $\A$ on paths constrained to remain in $\Sdif$ leads to \pref{euler-eq}.

\smallskip 

From now on, we shall consider Brenier's relaxed formulation \cite{brenier1989least}, \cite{brenier1993dual}, \cite{brenier1999minimal}, \cite{brenier2008generalized} of the minimizing geodesic problem between an initial and terminal configuration of the fluid. This formulation which allows splitting and crossing of particles, is based on the notion of generalized incompressible flow (GIF). Denoting  by $\Leb$ the Lebesgue measure on $\D$ (normalized so as to be a probability measure on $\D$), by $\Omega$ the path space 
\[\Omega:=C([0,T], \D)\]
and for $\omega\in \Omega$ and $t\in [0,T]$, the evaluation map at time $t$ is defined by $e_t(\omega):=\omega(t)$, the set of generalized incompressible flows is by definition the set of probability measures $Q$ on $\Omega$ such that ${e_t}_\#Q=\Leb$ for every $t\in [0,T]$,
\begin{equation}\label{defgif}
\GIF:=\{Q \in \PP(\Omega) \; : \; {e_t}_\#Q=\Leb, \;  \forall t\in [0,T]\}.
\end{equation}
We are also given $\pi_{0,T}\in \PP(\D\times \D)$ a probability measure on $\D\times \D$ having $\Leb$ as marginals and which captures the joint distribution of particles at times $0$ and $T$ (one may think for instance the deterministic coupling $\pi_{0,T}:=(\id, X_T)_\# \Leb$ where $X_T\in \Sdif$ represents the terminal Lagrangian configuration of the fluid).  The set of generalized incompressible flows  compatible with $\pi_{0,T}$ is then given by 
\begin{equation}
\GIF(\pi_{0,T}):=\{Q \in \GIF \; : \; ({e_0}, e_T)_\#Q=\pi_{0,T}\}.
\end{equation}
For $\omega\in  \Omega$ we denote by $E(\omega)$ its kinetic action:
\begin{equation}\label{defkine}
E(\omega):= \begin{cases} \frac{1}{2} \int_0^T \vert \dot{\omega} (t)\vert^2 \mbox{d} t \mbox{ if $\omega \in H^1((0,T), \D)$} \\ +\infty \mbox{ otherwise.} \end{cases}
\end{equation}
Brenier's relaxation of Arnold's geodesic problem then reads as the infinite-dimensional linear-programming problem
\begin{equation}\label{brerelax}
\inf_{Q \in \GIF(\pi_{0,T}) } \E(Q):=\int_{\Omega} E(\omega) \mbox{d} Q(\omega).
\end{equation}
This formulation can be viewed as an optimal transport problem with infinitely many marginal constraints corresponding to the incompressibility of the flow and an additional constraint corresponding to the prescribed  joint initial/terminal  distribution $\pi_{0,T}$.  It is probably the first instance of the nowadays active field of multi-marginal optimal transport \cite{pass2015multi}.

M\'erigot and Mirebeau \cite{memi}  recently produced a tractable numerical method for a non-convex Lagrangian formulation of  (\ref{brerelax}).   
The marginal constraints are penalized using   semi-discrete optimal transport for which fast solvers are now available (see \cite{merigotmulti}, \cite{levy3d}).

\smallskip

In the present paper, we follow a different approach, based on the so-called entropic regularization which leads to a strictly convex 
problem. The entropic regularization approach, which goes back to Schr\"{o}dinger \cite{Schrodinger31} has deep connections with large deviations \cite{DaG}. It has been extensively analyzed and developed by Mikami \cite{Mikami04} and  L\'eonard \cite{leonard2012schrodinger}, \cite{leonard2013survey}  who in particular proved convergence of Schr\"{o}dinger bridges to optimal transport geodesics as the noise intensity vanishes. It has also proved to be an efficient computational strategy for optimal transport by Cuturi \cite{Cut} who made the connection with the  simple but powerful Sinkhorn scaling algorithm which is equivalent to the famous iterative proportional fitting procedure (see \cite{Csi75}, \cite{Rus95}), also see  \cite{Ben} for various applications to optimal transport. 
 
Our aim in this paper is twofold: first  showing that the entropic regularization leads to tractable simulations. We present in section 6 computations of a non classical  two-dimensional GIF.  
Similar computations were presented recently in \cite{memi} but rely on a different relaxation and different numerical methods. Secondly, we connect this numerical scheme to the Schr\"{o}dinger bridge framework which  involves the relative entropy with respect to the Wiener measure with a small variance parameter.  We shall indeed show, that in the periodic case, $\D=\T^d$, our numerical approach is  exactly the discretized in time counterpart of the entropic interpolation approach developed recently in \cite{leonardBredinger} and reminiscent of Yasue's variational approach for Navier-Stokes equations \cite{yasue}.

\smallskip

The paper is organized as follows. After rewriting the time discretization of Brenier's problem as a multi-marginal optimal transport problem in section \ref{sec-time-discr}, we introduce its entropic regularization in section \ref{sec-entrreg}.  In the periodic case, a detailed comparison with the time-discretization of Bredinger's problem\footnote{we adopt here the terminology of \cite{leonardBredinger} where the name Bredinger is introduced as a contraction of Brenier and Schr\"{o}dinger.} as well as a $\Gamma$-convergence result are given in section \ref{sec-comp-bred}. In section \ref{sec-impl}, Sinkhorn's algorithm is described in the present setting. Finally, section \ref{sec-results} is devoted to numerical results in dimensions one and two.

%In section \ref{sec-entrreg} we start with the entropic approximation of Brenier's problem\footnote{The authors of \cite{leonardBredinger} call this entropic approximation the \emph{Bredinger} problem.} from \cite{leonardBredinger} and show that, once discretized in time, this problem is an entropic regularization of a  multi-marginal optimal transport problem. In section \ref{sec-impl}, we explain how a Sinkhorn-like algorithm can be implemented for this discretization. Section \ref{sec-results} presents numerical results. 

\section{Time discretization}\label{sec-time-discr}

In what follows, $\D$ is either a bounded convex domain of $\R^d$ or the flat torus, $\D=\R^d/2\pi \Z^d$, we denote by $\dist$ the distance on $\D$, that is the euclidean distance in the convex domain case and in the case of the torus:
\[\dist(x,y):=\inf_{k\in \Z^d} \vert x-y +2 \pi k\vert\, \; \forall (x,y)\in \T^d.\] 
The path space $\Omega=C([0,T], \D)$ is equipped with the topology of uniform convergence and $\PP(\Omega)$ with the corresponding weak $*$ topology. Given $N\in \N$, $N\ge 1$, let $\Tc_N:=\{k\frac{T}{N}, \; k=0, \ldots, N\}$ and consider the time-discretization of Brenier's problem \pref{brerelax}
\begin{equation}\label{brediscr}
\inf_{Q\in \GIF_N(\pi_{0,T}) } \E(Q),
\end{equation}
where 
\[\GIF_N(\pi_{0,T}):=\{Q \in \PP(\Omega) \; : \; {e_t}_\#Q=\Leb, \;  \forall t\in \Tc_N, \; (e_0, e_T)_\# Q=\pi_{0,T}\}.\]
It is well-known (see \cite{brenier1989least}) that both linear problems \pref{brerelax} and \pref{brediscr} admit solutions (which are not unique in general).

\smallskip

The discretized in-time problem \pref{brediscr} can easily be rewritten as an optimal transport problem with multi-marginal constraints as follows. Given $\xx_N:=(x_0, \cdots, x_N)\in \D^{N+1}$, let us denote by $\proj_{0,N}$ and $\proj_k$ the canonical projections:
\[\proj_k(\xx_N)=x_k, \; k=0, \cdots, N, \; \proj_{0,N}(\xx_N)=(x_0, x_N).\]
Defining the cost
\begin{equation}\label{defcn}
c_N(\xx_N):=\frac{N}{2T} \sum_{k=0}^{N-1} \dist^2(x_{k+1}, x_k), \; \forall  \xx_N:=(x_0, \cdots, x_N)\in \D^{N+1}
\end{equation}
and the set of plans 
\begin{equation}  \label{gammaN} 
\Gamma_N(\pi_{0,T}):=  \{\gamma \in \PP(\D^{N+1}) :  {\proj_k}_\# \gamma= \Leb, \; k=0,..., N, \; {\proj_{0,N}}_\# \gamma=\pi_{0,T}\}
\end{equation} 
let us consider the multi-marginal optimal transport problem:
\begin{equation}\label{mmbre}
\inf_{\gamma \in \Gamma_N(\pi_{0,T})} \int_{\D^{N+1}} c_N(\xx_N) \mbox{d} \gamma(\xx_N).
\end{equation}
Setting 
\[P_N(\omega):=(\omega(t))_{t\in \Tc_N},  \; \forall \omega \in \Omega\]
it is clear that $Q\in \PP(\Omega)$ belongs to $\GIF_N(\pi_{0,T})$ if and only if ${P_N}_\# Q$ belongs to $\Gamma_N(\pi_{0,T})$. Moreover, since 
\[c_N(x_0, \cdots, x_N)=\inf \Big\{E(\omega) \; : \;  \omega\in \Omega, \;  P_N(\omega)=(x_0, \cdots, x_N)\Big\},\]
we easily deduce the following:

\begin{lem}\label{equiv1}
Problems \pref{brediscr} and \pref{mmbre} are equivalent in the sense that $Q_N\in \GIF_N(\pi_{0,T})$ solves \pref{brediscr} if and only if ${P_N}_\# Q_N$ solves \pref{mmbre} and 
\[c_N(P_N(\omega))=E(\omega), \mbox{ for $Q_N$-a.e. $\omega$}.\]
\end{lem}

Finally, let us emphasize that \pref{brediscr} approximates \pref{brerelax} in the sense of $\Gamma$-convergence. Let us denote by  $\chi_{\GIF_N(\pi_{0,T})}$ and $ \chi_{\GIF(\pi_{0,T})}$ the characteristic function of $\GIF_N(\pi_{0,T})$ and $\GIF(\pi_{0,T})$ respectively i.e. for $Q\in \PP(\Omega)$, 
\[\begin{split}
\chi_{\GIF_N(\pi_{0,T})}(Q)& =\begin{cases} 0 \mbox{ if $Q \in \GIF_N(\pi_{0,T})$,}\\ +\infty \mbox{ otherwise, }\end{cases}, \\
 \chi_{\GIF(\pi_{0,T})}(Q) &=\begin{cases} 0 \mbox{ if $Q \in \GIF(\pi_{0,T}),$}\\ +\infty \mbox{ otherwise, }\end{cases},
\end{split}\] 
we indeed have (we refer to chapter 6 of \cite{thesislulu} for a proof):

\begin{prop}\label{gammacvdiscr}
The sequence of functionals on $\PP(\Omega)$ (endowed with the weak $^*$ topology),  $\E+ \chi_{\GIF_N(\pi_{0,T})}$ $\Gamma$-converges as $N\to +\infty$ to $\E+ \chi_{\GIF(\pi_{0,T})}$.
\end{prop}

Since $\E$ has relatively compact sublevel sets in $\Omega$, the previous result in particular implies convergence of minimizers of \pref{brediscr} to minimizers of \pref{brerelax}.

%\begin{proof}

%\end{proof}

\section{Entropy minimization}\label{sec-entrreg}

The equivalent linear programming problems \pref{mmbre} and \pref{brediscr} are extremely costly to solve numerically and a natural strategy, which has received a lot of attention recently, is to approximate these problems by  strictly convex ones by adding an entropic penalization.  First, let us  set a few notations, given a Polish space $X$, and two probability measures on $X$, $q$ and  $r$ the relative entropy of $q$ with respect to $r$ (a.k.a. Kullback-Leibler divergence) is given  by 
\[\H(q\vert r):=\begin{cases} \int_X \log\Big( \frac{ d q}{d r } \Big) \mbox{d} q \mbox{ if $q \ll r $}\\ + \infty \mbox{ otherwise}   \end{cases}\]
where $\frac{ d q}{d r} $ stands for the Radon-Nikodym derivative of $q$ with respect to $r$. If $X=\D^m$ and $\mu\in \PP(\D^m)$ we shall simply denote by $\Ent(\mu)$ the relative entropy of $\mu$ with respect to $\Leb^{ \otimes m}$; slightly abusing notation, when  $\mu \ll \Leb^{\otimes m}$ we shall identify $\mu$ with its density and simply write
\[\Ent(\mu)=  \frac{1}{\vert \D \vert^{m}}  \int_{\D^m} \log(\mu(x_1, \cdots, x_m)) \mu(x_1, \cdots, x_m) \mbox{d} x_1 \cdots \mbox{d} x_m\]
where $\vert \D \vert$ denotes the Lebesgue measure of $\D$.
\smallskip

A first  way to perform an entropic regularization of \pref{mmbre} is, given a small parameter $\eps>0$, to replace \pref{mmbre} by 
\begin{equation}\label{entrnaive}
\inf_{\gamma \in \Gamma_N(\pi_{0,T})}   \int_{\D^{N+1}} c_N \mbox{ d} \gamma+ \eps \Ent(\gamma)
\end{equation} 
Note that for the previous problem to make sense, i.e. for the existence of at least one $\gamma\in \Gamma_N(\pi_{0,T})$ such that $\Ent(\gamma)<+\infty$ it is necessary and sufficient that
\begin{equation}\label{conditionentfin}
\Ent(\pi_{0,T}) <+\infty.
\end{equation}
This, in particular, rules out the case\footnote{However, as we shall see, one way to overcome this problem is by penalizing in the cost the condition $x_N=X_T(x_0)$.} where $\pi_{0,T}=(\id, X_T)_\# \Leb$ with $X_T\in \Sdif$. Defining the Gibbs measure on $\D^{N+1}$ associated to the cost $c_N$:
\[\eta_{N, \eps} (\xx_N):= \Lambda_{N, \eps} \exp\Big(-\frac{c_N(\xx_N)}{\eps}\Big) \]
where $\Lambda_{N, \eps}$ is the normalizing constant which makes $\eta_{N, \eps}$ be a probability density, \pref{entrnaive} can be rewritten as the Kullback-Leibler projection problem of $\eta_{N, \eps}$ onto  $\Gamma_N(\pi_{0,T})$:
\begin{equation}\label{entrnaivekl}
\inf_{\gamma \in \Gamma_N(\pi_{0,T})}   \eps \H(\gamma \vert \eta_{N, \eps}). 
\end{equation} 

It will follow from the proof of Theorem \ref{gammaconveps} (also see remark \ref{gammaconvs}) that in the periodic case $\D=\T^d$, the finite entropy condition \pref{conditionentfin} guarantees the $\Gamma$-convergence of \pref{entrnaive} to \pref{mmbre} as $\eps\to 0$.

%\begin{thm}\label{gammaconvnaive}
%Let us assume the finite entropy condition \pref{conditionentfin}, then $J_{N, \eps}$ $\Gamma$-converges to $J_N$ as $\eps \to 0^+$ for   the weak$\textendash^*$ topology of $\PP(\D^{N+1})$. 
%\end{thm}

\section{Comparison with Bredinger in the periodic case}\label{sec-comp-bred}

\subsection{Bredinger's problem and its time discretization}

Another way to approximate the problem is to introduce some noise at the level of the path space i.e. in \pref{brerelax} as was done recently in \cite{leonardBredinger}\footnote{in connection with Navier-Stokes.}  as follows. Throughout this section, we assume that $\D=\T^d:=\R^d/2\pi \Z^d$. Assuming \pref{conditionentfin} and given a small parameter $\eps>0$ as before, let $R_\eps\in \PP(\Omega)$ be defined by
\begin{equation}\label{defRefep}
R_\eps=  \frac{1}{(2\pi)^d} \int_{\T^d} \Law(x+\sqrt{\eps} B_\eps) \mbox{d}x
\end{equation}
where $B_\eps$ is the standard Brownian motion starting at $0$ (that is the Markov process whose generator  is $\frac{1}{2} \Delta$) on $\frac{1}{\sqrt{\eps}}\T^d$. Following \cite{leonardBredinger}, we shall call 
\begin{equation}\label{bred}
\inf_{Q\in \GIF(\pi_{0,T})} \H(Q\vert R_\eps)
\end{equation}
the Bredinger problem (with variance parameter $\eps$). Let us now discretize in time the Bredinger problem \pref{bred} in a similar way as we deduced \pref{brediscr} from \pref{brerelax} i.e. consider
\begin{equation}\label{breddiscr}
\inf_{Q\in \GIF_N(\pi_{0,T})} \H(Q\vert R_\eps).
\end{equation}
It is worth noting here that $R_\eps\in \GIF$. Following L\'eonard \cite{leonard2012schrodinger}, one can reduce \pref{breddiscr} to an entropy minimization problem over $\Gamma_N(\pi_{0,T})$ as follows. Let us set 
\begin{equation}\label{deftet}
\theta_{N, \eps}:={P_N }_\# R_\eps
\end{equation}
and disintegrate $R_\eps$ with respect to $\theta_{N, \eps}$ as 
\[R_{\eps} = \int_{{(\T^d)}^{N+1}} R_\eps^{\xx_N} \mbox{d} \theta_{N, \eps}(\xx_N),\]
so that $R_\eps^{\xx_N}$ is the law of a Brownian bridge i.e. the law of a Brownian path conditional to the fact that its values at times in $\Tc_N$ are given by $\xx_N$. 
In a similar way, given $Q\in \GIF_N(\pi_{0,T})$, setting $\gamma:={P_N}_\# Q \in \Gamma_{N}(\pi_{0,T})$ and disintegrating $Q$ with respect to $\gamma$:
\[Q = \int_{{(\T^d)}^{N+1}} Q^{\xx_N} \mbox{d} \gamma (\xx_N),\]
we have
\[\H(Q\vert R_\eps)=\H(\gamma\vert \theta_{N, \eps})+ \int_{(\T^d)^{N+1}} \H(Q^{\xx_N}\vert R_{\eps}^{\xx_N}) \mbox{d} \gamma(\xx_N) \ge \H(\gamma\vert \theta_{N, \eps}) \]
with equality if and only if $Q^{\xx_N}=R_{\eps}^{\xx_N}$ for $\gamma$-almost every $\xx_N$ . Hence we get the following entropic analogue of Lemma \ref{equiv1}: 
%which can be seen as a generalization
%to the $N$-marginal case of F\"{o}llmer's theorem:

\begin{prop}\label{equiv2}
Let $Q^\star\in \GIF_N(\pi_{0,T})$ then $Q^\star$ solves \pref{breddiscr} if and only if $\gamma^\star:={P_N}_\# Q $ solves
\begin{equation}
\label{mmkent}
\inf_{\gamma \in \Gamma_N(\pi_{0,T})} \H(\gamma \vert \theta_{N, \eps}) 
\end{equation}
and $Q^\star$ disintegrates with respect to $\gamma^\star:={P_N}_\# Q^\star $  as
\begin{equation}
\label{same-bridge}
Q^\star = \int_{{\T^d}^{N+1}} R_\eps^{\xx_N} \mbox{d} \gamma^\star (\xx_N).
\end{equation}
\end{prop}
%\begin{rem}
% Before giving the proof of proposition \ref{equiv2}, let us notice that equation \pref{same-bridge} means that the optimal measure $Q^\star$ and the reference measure $R_\eps$ share the same bridges.
%\end{rem}

\begin{proof}
 By strict convexity,  \pref{breddiscr} and \pref{mmkent}  admit at most one solution. Thanks to 
\begin{equation}
\label{additive-formula}
\H(Q\vert R_\eps)=\H(\gamma\vert \theta_{N, \eps})+ \int_{(\T^d)^{N+1}} \H(Q^{\xx_N}\vert R_{\eps}^{\xx_N}) \mbox{d} \gamma(\xx_N)
\end{equation}
where $\gamma:={P_N}_\# Q$ and the fact that $Q\in\GIF_N(\pi_{0,T})$ if and only if $\gamma:={P_N}_\# Q\in \Gamma_N(\pi_{0,T})$
the minimization problem \pref{breddiscr} can be rewritten as
\begin{equation}
%\begin{split}
%&\inf\left\{ \H(Q|R_\eps)\;|\;Q\in\GIF_N(\pi_{0,T})\right\} =\\ 
\inf_{ \gamma\in\Gamma_N(\pi_{0,T})}     \left\{ \inf  \left\{\H(Q|R_\eps)\;|\;Q\in\PP(\Om),\;{P_N}_\#Q=\gamma  \right\} \right\}.
%\end{split}
\end{equation}
With (\ref{additive-formula}), the inner minimization is uniquely solved when $Q^{\xx_N}=R_{\eps}^{\xx_N}$ for $\g-$almost every $\xx_N\in{\T^d}^{N+1}$ since $\H(Q^{\xx_N}\vert R_{\eps}^{\xx_N})=0$ is the minimal
value of the relative entropy. Therefore, for each $\gamma\in\Gamma_N(\pi_{0,T})$ we have
\[  \inf  \left\{\H(Q|R_\eps)\;|\;Q\in\PP(\Om),\;{P_N}_\#Q=\gamma  \right\}
=\H(Q^\gamma|R_{\eps}^{\xx_N})=\H(\gamma|\theta_{N,\eps}),
\]
where 
\[ Q^\gamma=\int_{{(\T^d)}^{N+1}}R_{\eps}^{\xx_N} \mbox{d} \gamma(\xx_N), \]
and the solution of \pref{breddiscr} is
\[Q=Q^{\gamma^\star} \]
where $\g^\star$ is the unique solution of \pref{mmkent}.
\end{proof}

Now, we see that \pref{mmkent} leads to another entropy regularized optimal transport problem, which can equivalently be rewritten as
\begin{equation}\label{bresinkh}
\inf_{\gamma \in \Gamma_N(\pi_{0,T})}   \int_{(\T^d)^{N+1}} c_{N, \eps} \mbox{ d} \gamma+ \eps \Ent(\gamma) \mbox{ with } c_{N, \eps}:=-\eps \log(\theta_{N, \eps}).
\end{equation}

The difference between the \emph{naive} regularization \pref{entrnaive} and \pref{bresinkh} is that the cost which comes from the discretization of Bredinger's problem $c_{N, \eps}$ is related to the heat kernel and not directly to the initial quadratic distance cost. We shall compare both costs and give a convergence  result in the next paragraph. We conclude this section by summarizing in the following table all the variational models we have seen so far.

 %but let us indicate that using $c_{N, \eps}$ and not $c_N$ is actually very important in practice when using Sinkhorn algorithm because Sinkhorn steps mainly amounts to performing successive integrations with respect to the heat kernel  and it  can be done quite efficiently in Fourier. \\

\TabOne{

% \hline
% \textbf{Arnold's Principle} \\\hline
%  \[\inf\left\{\A(X)\;|\;t\mapsto X(t,\cdot)\in\Sdif\right\}\]
%   where \[\A(X)=\int_0^T \Vert \dot{X}\Vert^2_{L^2(\D)} \mbox{d} t.\]\\
\hline 
 \textbf{Brenier's problem} \\\hline
 \[\inf_{Q\in\GIF(\pi_{0,T})} \E(Q)\]
 where \[\E(Q)\ddd\int_{\Om}E(\omega)\mbox{d} Q(\omega).\]\\
\hline 
 \textbf{Discrete Brenier problem} \\\hline
 \[\inf_{Q\in\GIF_N(\pi_{0,T})} \E(Q)   \Longleftrightarrow \inf_{\gamma \in \Gamma_N(\pi_{0,T})} \int_{(\T^d)^{N+1}} c_N(\xx_N) \mbox{d} \gamma(\xx_N).\]
% \[ \inf\left\{ \Ee_K(\g)\;|\;\eta\in\TPK  \right\}, \]
% where $\Ee_K(\g)=\int_{\Om}E_K(\om)d\g(\om)$ and $E_K(\om)=\sum_{i=1}^K\dfrac{K}{2}|\om(t_i)-\om(t_{i-1})|^2$.
\\
\hline 
 \textbf{Bredinger's problem} \\\hline  
 \[\inf_{Q\in \GIF(\pi_{0,T})} \H(Q\vert R_\eps) \]
% \[\inf\left\{ \H(P|R_\eps)\;|\;P\in\P(\Om),\;(e_t)_\sharp P=\Ll_\D,\; (e_0,e_1)_\sharp P=\tilde{P}\right\}, \]
% where $\H(P|R_\eps)$ is the relative entropy, $\tilde{P}=\tilde{\g}$, $R_\eps$ a Brownian motion with variance $\eps$ and $\eps$ can be interpreted as a viscosity parameter.
\\
\hline 
\textbf{Discrete Bredinger problem} \\\hline
\[\inf_{Q\in \GIF_N(\pi_{0,T})} \H(Q\vert R_\eps) \Longleftrightarrow \inf_{\gamma \in \Gamma_N(\pi_{0,T})} \H(\gamma \vert \theta_{N, \eps}). \]
%\[\inf\left\{ \H(P|R_\eps)\;|\;P\in\TPK\right\}=\inf\left\{ \H(\g|\eta_\eps)\;|\;\g\in\TPk\right\},\]
%where \[\eta_\eps=\dfrac{1}{(2\pi\eps/K)^{dK/2}}\exp{(-\dfrac{K}{2\eps}\sum_{i=1}^K|x_i-x_{i-1}|^2)}dx_0\cdots dx_K.\]
\\
\hline
}

 %\subsection{Gaussian estimates for the heat kernel}\label{sec-gaussian-comp}

\subsection{Convergence as noise vanishes }\label{sec-conveps}

Our goal now is to establish (for fixed $N$) a convergence result as $\eps\to 0$. In the first place, one needs to  connect  $c_{N, \eps}$ and $c_N$ which can be done thanks to classical Gaussian estimates for the heat kernel. 

First we observe that the kernel $\theta_{N, \eps}$ from \pref{defRefep}-\pref{deftet} can be computed as follows. First let us denote by $p_t$ the heat kernel on $\R^d$:
\begin{equation}\label{hk}
p_t(z):=\frac{1}{(2\pi t)^{\frac{d}{2}} }\exp\Big(-\frac{\vert z \vert^2}{2t}\Big), \; t>0, \; z\in \R^d,
\end{equation}
the heat-kernel on $\T^d$ is obtained by its $2\pi \Z^d$ periodization:
\begin{equation}
g_t(x):= (2\pi)^d \sum_{k\in \Z^d} p_t(x+ 2 k\pi), \;  x\in \T^d \simeq [-\pi, \pi]^d,
\end{equation}
and the heat kernel on  $\frac{1}{\sqrt \eps}\T^d$ is likewise given by 
\begin{equation}
g_t^\eps(x):= (2\pi)^d \sum_{k\in \Z^d} p_t(x+ \frac{2 k\pi}{\sqrt{\eps}}), \;  x\in  \frac{1}{\sqrt \eps}\T^d.
\end{equation}
Since 
\[\eps^{-\frac{d}{2}} g_t^\eps  \Big(\frac{x}{\sqrt{\eps}}\Big)=g_{\eps t} (x), \;  x\in \T^d,\]
we have  
\[\theta_{N, \eps}(x_0, \cdots, x_N)=  \eps^{-\frac{dN}{2}} \prod_{k=0}^{N-1} g^\eps_{\frac{ T}{N}}\Big(\frac{x_{k+1}-x_k}{\sqrt{\eps}}\Big)  = \prod_{k=0}^{N-1} g_{\frac{\eps T}{N}}(x_{k+1}-x_k)\] 
so that
\[c_{N,\eps} (x_0, \cdots, x_N)=-\eps \sum_{k=0}^{N-1} \log\Big(g_{\frac{\eps T}{N}}(x_{k+1}-x_k) \Big).\]
To compare this expression with the minimal quadratic cost $c_N$, a natural tool is the following Gaussian estimate for the heat kernel (see the self-contained notes of \cite{mah} for the case of the torus):

\begin{thm}\label{gaussianest}
The heat kernel $g_t$ on the torus satisfies for every $t>0$ and $x\in \T^d$:
\[\frac{ \lambda }{(2\pi t)^{\frac{d}{2}} }\exp\Big(-\frac{\dist^2(x,0)}{2t}\Big)     \le    g_t(x) \le   \frac{\Lambda}{(2\pi t)^{\frac{d}{2}} }\exp\Big(-\frac{\dist^2(x,0)}{2t}\Big)\]

where $\lambda$ and $\Lambda$ are two positive constants (depending on $d$).

\end{thm}

Taking logarithms and summing over $k$ we thus obtain
\[-N\eps \log(\Lambda)+\frac{dN \eps}{2} \log\Big(\frac{2 \pi \eps T}{N} \Big) \le   c_{N, \eps}-c_N \le  -N\eps \log(\lambda)+\frac{dN \eps}{2} \log\Big(\frac{2 \pi \eps T}{N} \Big)\]
i.e. there is a constant $M$ such that, for small enough $\eps$, one has
\begin{equation}\label{convergencecost}
\vert c_N -c_{N, \eps} \vert \le M (\eps N \vert \log(\frac{\eps T}{N})\vert). 
\end{equation}

%\subsection{A $\Gamma$-convergence result}\label{sec-gammaeps}

From \pref{convergencecost},  we can easily deduce a $\Gamma$-convergence result as $\eps\to 0$ (and $N$ fixed) between \pref{mmbre} (equivalent to \pref{brediscr}) and its entropic regularization \`a la Bredinger \pref{bresinkh} (equivalent to \pref{breddiscr}). For $\gamma\in \PP((\T^d)^{N+1})$, let us denote by $\chi_{\Gamma_N(\pi_{0,T})}$ the characteristic function of $\Gamma_N(\pi_{0,T})$ i.e. 
\[\chi_{\Gamma_N(\pi_{0,T})}(\gamma)=\begin{cases} 0 \mbox{ if $\gamma \in \Gamma_N(\pi_{0,T})$}\\ +\infty \mbox{ otherwise }\end{cases}\]
and define the functionals to be minimized in \pref{mmbre} and \pref{bresinkh} respectively
\[J_N(\gamma):=\int_{(\T^d)^{N+1}} c_N \mbox{d} \gamma+ \chi_{\Gamma_N(\pi_{0,T})}(\gamma),\]
and
\[%\begin{split}
J_{N, \eps}(\gamma):=\int_{(\T^d)^{N+1}} c_{N, \eps} \mbox{d} \gamma+  \eps \Ent(\gamma)+\chi_{\Gamma_N(\pi_{0,T})}(\gamma)%\\
= \eps \H(\gamma\vert \theta_{N, \eps})+\chi_{\Gamma_N(\pi_{0,T})}(\gamma). 
%\end{split}
\]

We then have

\begin{thm}\label{gammaconveps}
Let us assume the finite entropy condition \pref{conditionentfin}, then $J_{N, \eps}$, $\Gamma$-converges to $J_N$ as $\eps \to 0^+$ for   the weak $*$ topology of $\PP((\T^d)^{N+1})$. 
\end{thm}

\begin{proof}
Let $\gamma_\eps$ converge weakly $*$ to $\gamma$, we first have to prove the $\Gamma$-liminf inequality
\begin{equation}\label{galimi}
\liminf_{\eps\to 0} J_{N, \eps} (\gamma_\eps) \ge J_N(\gamma).
\end{equation}
We may assume that $\gamma_\eps\in \Gamma_N (\pi_{0,T})$ for a vanishing sequence of $\eps$ otherwise there is nothing to prove. Since $ \Gamma_N (\pi_{0,T})$ is weakly $*$ closed we then also have $\gamma \in  \Gamma_N (\pi_{0,T})$. Using the fact that for $\gamma_\eps \in   \Gamma_N (\pi_{0,T})$, $\Ent(\gamma_\eps) \ge \Ent(\Leb^{N+1})= (N+1) \Ent(\Leb)$ and the estimate \pref{convergencecost} we get
\[J_{N, \eps} (\gamma_\eps)  \ge -M (\eps N \vert \log(\frac{\eps T}{N})\vert)+ \int_{(\T^d)^{N+1}} c_N \mbox{d} \gamma_\eps   + (N+1) \eps  \Ent(\Leb) \]
and since $c_N$ is continuous, \pref{galimi} immediately follows.

\smallskip

As for the $\Gamma$-limsup inequality, we have, given $\gamma\in \Gamma_N (\pi_{0,T})$, to find $\gamma_\eps\in  \Gamma_N (\pi_{0,T})$ converging weakly $*$ to $\gamma$ and such that
\begin{equation}\label{galims}
\limsup_{\eps\to 0} \int_{(\T^d)^{N+1}} c_{N, \eps} \mbox{d} \gamma_\eps +\eps \Ent(\gamma_\eps) \le \int_{(\T^d)^{N+1}} c_N \mbox{d} \gamma.
\end{equation}
We approximate $\gamma\in \Gamma_N (\pi_{0,T})$ by $\gamma_\eps$ as follows. First disintegrate $\gamma$ with respect to its projection on the first and last components as 
\[\gamma= \gamma^{x_0, x_N} \otimes \pi_{0,T}\]
and set 
\[\gamma_\eps=\gamma_\eps^{x_0, x_N} \otimes \pi_{0,T} \mbox{ with }  \gamma_\eps^{x_0, x_N}=g_\eps^{\otimes(N-1)} \star \gamma^{x_0, x_N}.\]
By construction ${\proj_{0, N}}_\# \gamma_\eps=\pi_{0,T}$ and ${\proj_k}_\# \gamma_\eps= g_\eps \star \Leb =\Leb$ since $\Leb$ is invariant by the heat flow, we thus have $\gamma_\eps \in \Gamma_N(\pi_{0,T})$ and $\gamma_\eps$ converges to $\gamma$ as $\eps\to 0$. To estimate $\Ent(\gamma_\eps)$, we first observe that
\[\Ent(\gamma_\eps)=\Ent(\pi_{0,T}) +\int_{\T^d \times \T^d} \Ent(\gamma_\eps^{x_0,x_N}) \mbox{d} \pi_{0,T}(x_0, x_N)\]
and then, by construction of $\gamma_\eps^{x_0, x_N}$ and the convexity of $\Ent$ 
 \[ \begin{split}
 \Ent(\gamma_\eps^{x_0,x_N}) & =   \Ent\Big( \int_{{(\T^d)}^{N-1}} \prod_{k=0}^{N-1} g_\eps(x_k- \cdot) \mbox{d} \gamma^{x_0, x_N}(x_1, \cdots, x_{N-1})  \Big) \\
 &\le \int_{{(\T^d)}^{N-1}} \Ent( \prod_{k=0}^{N-1} g_\eps(x_k- \cdot))  \mbox{d} \gamma^{x_0, x_N}(x_1, \cdots, x_{N-1})\\
      &= \Ent( g_\eps^{\otimes(N-1)}) =(N-1) \Ent(g_\eps)
 \end{split}\]
 thanks to Theorem \ref{gaussianest} we have
 \[  \Ent(g_\eps) \le \log(\Lambda)-\frac{d}{2} \log(2\pi \eps)\]
 so, thanks to \pref{conditionentfin}, we get $\Ent(\gamma_\eps)=O(\vert \log(\eps)\vert)$ hence $\limsup_{\eps \to 0} \eps \Ent(\gamma_\eps)\le 0$, thanks to \pref{convergencecost} we thus deduce \pref{galims}. 

%put differently, if $\varphi\in C({\T^d}^{N+1})$,
%\[\int_{\T^d}^{N+1}} \varphi \mbox{d}\gamma_\eps= \int_{\T^d\times \T^d}  \varphi(x_0, x_1-z_1, x_2-z_2, \cdotsn x_{N-1}-z_{N-1}, x_N) \prod_{} g_\eps(z_i) dz_i \mbox{d} \gamma^{x_0, x_N}(x_1, \cdots, x_N{-1})  \Big(\Big) \mbox{d}\pi_{0,T}(x_0, x_N)\]

\end{proof}

\begin{rem}\label{gammaconvs}
Of course,  in the periodic case, there is also $\Gamma$-convergence as $\eps \to 0^+$ if one uses $c_N$ instead of $c_{N, \eps}$ i.e. when considering \pref{entrnaive} instead of  \pref{bresinkh}.  The case of a domain with boundary is less clear to us, in particular the $\Gamma$-limsup argument by convolution with the heat flow above does not work in this case.

\end{rem}

\section{Sinkhorn algorithm}\label{sec-impl}

Once $\eps$ and $N$ are fixed, both entropic regularizations \pref{entrnaive} (written as in \pref{entrnaivekl}) and \pref{mmkent} can be written in a common way. It reads as a Kullback-Leibler projection on the set $\Gamma_N(\pi_{0,T})$, of a certain reference measure $\alpha$ whose density is a product kernel:
\begin{equation}\label{projkl}
\inf_{\gamma \in \Gamma_N(\pi_{0,T})} \H(\gamma \vert \alpha ) \mbox{ with } \alpha(x_0, \cdots x_N):=\prod_{k=0}^{N-1} K(x_{k+1}-x_k)
\end{equation}
where, in the case of  \pref{entrnaive}, $K$ is the Gaussian kernel 
\[K(x)= \exp\Big(-\frac{ N \dist^2(x,0) }{2\eps T}\Big)\]
(we have omitted the normalizing constant which does play any role in the minimization problem) and in the Bredinger case \pref{mmkent}, $K$ is the heat kernel
\[K(x)=g_{\frac{\eps T}{N}}(x).\]
There is a unique solution to \pref{projkl} which is of the form
\begin{equation}\label{defgamaab}
\gamma(x_0, \cdots, x_N)=\gamma_{a,b}(x_0, \cdots x_N):=  b(x_0, x_N) \prod_{k=1}^{N-1} a_k(x_k) \alpha(x_0, \cdots x_N)
\end{equation}
where $b=b(x_0,x_N)$ and $a=(a_1(x_1), \cdots, a_{N-1}(x_{N-1}))$ are positive potentials (exponentials of the Lagrange multipliers\footnote{this is formal  since existence of Lagrange multipliers for the continuous problem cannot be taken for granted in infinite dimensions unless a demanding qualification-like assumption is met requiring that $\pi_{0,T}$ somehow lies in the interior of the  domain of the relative entropy. Nevertheless, once discretized in space, the problem becomes a finite-dimensional smooth convex  minimization with linear constraints so that the existence of such multipliers is guaranteed. To reduce the amount of notation here, we use the same notations for the continuous problem \pref{projkl} as for the discretized one where  integrals are replaced by finite sums.}    associated to the marginal constraints). These positive potentials are (uniquely up to multiplicative constants with unit product) determined by the requirement that $\gamma_{a, b} \in \Gamma_N(\pi_{0,T})$ i.e. the relations
\begin{equation*}\label{marge0Tok} 
\pi_{0,T}(x_0, x_N)=b(x_0, x_N) \int_{\D^{N-1}} \prod_{k=1}^{N-1} a_k(x_k)   \alpha(x_0, \cdots x_N) \mbox{d} x_1 \cdots \mbox{d} x_{N-1}
\end{equation*}
and for  $k=1, \cdots, N-1$
\begin{equation*}\label{margekok} 
\frac{1}{ \vert \D \vert} = a_k(x_k) \int_{\D^{N}} \prod_{j=1,  j\neq k}^{N-1} a_j(x_j) b(x_0, x_N)  \alpha(x_0, \cdots x_N) \mbox{d} x_0  \cdots \mbox{d} x_{k-1}  \mbox{d} x_{k+1} \cdots \mbox{d} x_N.
\end{equation*}
The idea of Sinkhorn algorithm (also known as IPFP, \emph{Iterated Proportional Fitting procedure}) is to update one multiplier so as to fit one marginal constraint at a time. The algorithm constructs inductively a sequence of measures with densities as follows. Start with
\[\gamma^{(0)} =\alpha=\gamma_{a^{(0)}, b^{(0)}}, \; a^{(0)}=(1, \cdots, 1), \; b^{(0)}=1\]  
and once $\gamma^{(l)}=\gamma_{a^{(l)}, b^{(l)}}$ (with $a^{(l)}=(a_1^{(l)}, \cdots, a_{N-1}^{(l)})$) has been determined, compute $b^{(l+1)}$ by 
\begin{equation}\label{marge0Tokl} 
b^{(l+1)}(x_0, x_N)  = \dfrac{\pi_{0,T}(x_0, x_N)}{ \displaystyle{\int_{\D^{N-1}} \prod_{k=0}^{N-1} a_k^{(l)}(x_k)   \alpha(x_0, \cdots x_N) \mbox{d} x_1 \cdots \mbox{d} x_{N-1}}}
\end{equation}
and then  
\[a_k^{(l+1)}(x_k)=\frac{1 } {\vert \D\vert A_k^{(l+1)}(x_k)}, k=1, \cdots, N-1, x_k \in \D\]
with
\begin{equation}\label{marge1okl} 
 A_1^{(l+1)} (x_1)   =  \int_{\D^{N}} \prod_{j=2 }^{N-1} a_j^{(l)} (x_j) b^{(l)}(x_0, x_N)   \alpha(x_0, \cdots x_N) \mbox{d} x_0 \mbox{d} x_2 \cdots \mbox{d} x_N
\end{equation}
and  for $k=2, \cdots ,N-1$, setting $\xx_{-k}:=(x_0, \cdots, x_{k-1}, x_{k+1}, \cdots, x_N)$:
\begin{equation}\label{marge2okl} 
A_k^{(l+1)} (x_k)   =  \int_{\D^{N}} \prod_{j=1}^{k-1} a_j^{(l+1)} (x_j)      \prod_{j=k+1}^{N-1} a_j^{(l)} (x_j)     b^{(l)}(x_0, x_N)   \alpha (x_0, \cdots x_N) \mbox{d} \xx_{-k}     
\end{equation}

Finally,  set $\gamma^{(l+1)}=\gamma_{a^{(l+1)}, b^{(l+1)}}$ with $a^{(l+1)}=(a_1^{(l+1)}, \cdots, a_{N-1}^{(l+1)})$. 
This algorithm can be viewed as Kullback-Leibler projecting $\alpha$ in an alternating way onto each linear marginal constraint \cite{Ben}. In finite dimensions (hence for the discretized problem) it is well-known (see \cite{bauschke-lewis}) that this algorithm converges to the projection onto the intersection i.e. $\Gamma_N(\pi_{0,T})$. 

\smallskip

\begin{rem}[Implementation] 
The iterations of Sinkhorn might seem tedious at a first glance because of the integration against $\alpha$, but due to the special form of $\alpha$, these are just series of convolution with the kernel $K$. In the periodic case, this roughly amounts to compute efficiently Fourier coefficients. In addition, a useful property in practice is that the kernel $K$ can be written in tensorized form
\[K(z_1, \cdots, z_d)=\prod_{j=1}^d k(z_j)\] 
where $k$ is a kernel in dimension one.  
\end{rem} 
 
 \begin{rem}[Penalization of the terminal configuration]
When the Lagrangian coupling between initial and final configuration is deterministic, i.e.  $\pi_{0,T}:=(\id, X_T)_\# \Leb$, it is possible to 
take into account this constraint as a penalization~: first the constraint  ${\proj_{0,N}}_\# \gamma=\pi_{0,T}$ is removed from 
(\ref{gammaN}) which therefore does not depend anymore on $\pi_{0,T}$, secondly a least square penalization with a positive parameter $\beta$ 
is added to (\ref{defcn}) which becomes 
\begin{equation*}
c_{N,\beta}(\xx_N):=\frac{N}{2T} \sum_{k=0}^{N-1} \dist^2(x_{k+1}, x_k)  +  \beta\, \dist^2(x_{N}, X_T(x_0))% , \; %\forall  \xx_N:=(x_0, \cdots, x_N)\in \D^{N+1}  .
\end{equation*}
The cost now depends on $X_T$ and $\beta$.   

\smallskip

The impact on Sinkhorn algorithm is  small and the complexity unchanged~:
The Lagrange multiplier $b$ disappears, take $b \equiv 1$ in (\ref{marge0Tokl}-\ref{marge1okl}-\ref{marge2okl}). 
In the same equations, the kernel $\alpha$  (\ref{projkl}) becomes also in (\ref{marge0Tokl}-\ref{marge1okl}-\ref{marge2okl}) 
\begin{equation*}
\alpha_\beta (x_0, \cdots x_N) =  \alpha (x_0, \cdots x_N) \, \exp\Big(-\frac{ \beta  \vert x_{N} - X_T(x_0)   \vert^2 }{2\eps }\Big)
\end{equation*}
Of course, $\beta$ has to be dimensionalized according to  $\frac{T}{N}$.
  \end{rem}

\section{Numerical results}\label{sec-results}

\subsection{One-dimensional  experiments}
% Figures~\ref{fig:mm1}, \ref{fig:mm2} and~\ref{fig:mm3} show $\g_{0,k}$ for three test cases in dimension $d=1$ proposed in~\cite{brenier2008generalized}. The computation is performed with a uniform discretization $(x_i)_i$ of $[0,1]$ with $N=200$ points, $\eps=10^{-3}$ and $N+1=16$. They agree with the solutions produced by Brenier and the mass spreading of the generalized flow is nicely captured by the 2 marginals couplings~\eqref{2coupl}. 

We first reproduce a 1D result obtained with the method described in section 5 and taken from \cite{Ben}. 
It is a simulation of a test case proposed (and solved with a Lagrangian method) in \cite{brenier1989least}.
It provides a good warm up for the presentation of  2D results  which are new.  \\

The first test case is {\em not periodic}, it is set on $ {\cal D } = [0,1] $ equipped with the standard Euclidean distance.
The final configuration is given by $ X_T(x_0) = 1 -x_0$ which 
is not an orientation preserving diffeomorphism.  Arnold problem (\ref{arnold}) does not make sense since  {\em classical} particles cannot cross. \\

We apply the algorithm of section 5 to Brenier $\GIF$ relaxation.  The multimarginal transport plans gives a Eulerian 
measure of the mass movements.  In order to  track the underlying Lagrangian motion  we 
represent in Figure \ref{fig:mm3} and  for different times $t_k$   
the quantity 
\begin{equation} 
\label{ff1} 
P_{t_k}(x_0,x_k) = \int_{\D^{N-2}}  \gamma(  {x_0}   \cdots   x_N  ) \mbox{d} x_1 \cdots \mbox{d} x_{k-1} \mbox{d} x_{k+1} \cdots \mbox{d} x_N,
\end{equation} 
which represents  the amount of mass which has traveled from $x_0$ to $x_k$ between initial time and time $t_k$. If the solution was
classical and deterministic (but it is not), it would correspond to $(\id,X_{t_k})_\# \Leb$. \\

The initial and final correlations  $ (\id,\id)_{\#} \Leb$ and  $ (\id,X_T)_{\#} \Leb$ are prescribed.
All marginals  are the Lebesgue measure.   The solution forces  the mass to split and  the  ``generalized particles'' to mix and cross to minimize the  kinetic energy of the GIF.   

The results, displayed in figure \ref{fig:mm3} are consistent with the 1989 Lagrangian simulation of Brenier \cite{brenier1989least}.   
For this case we also plotted in figure \ref{fig:it-hil} the the Hilbert projective metric between two consecutive iterations of Sinkhorn and for three different number of marginals.
The Hilbert metric between two strictly positive vectors $p,q\in\R_{++}^M$ is defined as
\begin{equation}
 \Hil:=\log\bigg (\dfrac{\max_i \frac{p_i}{q_i}}{\min_i\frac{p_i}{q_i}} \bigg).
\end{equation}
Since in this case $N+1$ marginals are considered and so $N+1$ Lagrange multipliers $a_i$, the Hilbert metric on the product space can be written as the sum of $N+1$ terms: if we consider the variables $a_i$ at step $l$ and $l+1$ the Hilbert metric is written as
\begin{equation}
\Hil_k^N(\otimes_{i=0}^N a^{(l+1)}_{i},\otimes_{i=0}^Na^{(l)}_{i}):= \sum_{i=0}^N \Hil(a^{(l+1)}_{i},a_{i}^{(l)}).
\end{equation}
It has been showed in \cite{FRANKLIN1989717,georgiou2015positive} (and for a generalization of Sinkhorn in \cite{chizat2016scaling}) that the map defined by Sinkhorn iterations is indeed contractive
in the Hilbert projective metric for the two-marginals case. As one can see in figure \ref{fig:it-hil} this is still true in the multi-marginal generalization of Sinkhorn, even if it seems that the number of iterations required to convergence
(for the three computations we have used $\Hil_k^N$ as stopping criteria with a tolerance $\eta=10^{-4}$) depends on the number of marginals (or time steps) chosen.
\\

%{\bf LUCA give parameters of discretzation run times ? and value of $\epsilon$ } 

\newcommand{\MyFigEuler}[2]{\imgbox{\includegraphics[width=.19\linewidth]{EulerT#1timeStep#2}}}

% \begin{figure}[h!]
% 	\centering
% 	\TabFive{
% 		\MyFigEuler{1}{1} &
% 		\MyFigEuler{1}{2} &
% 		\MyFigEuler{1}{3} &
% 		\MyFigEuler{1}{4} &
% 		\MyFigEuler{1}{5} \\
% 		$t=0$ & $t=1/8$ & $t=1/4$ & $t=3/8$ & $t=1/2$
% 	}
% 	\TabFour{
% 		\MyFigEuler{1}{6} &
% 		\MyFigEuler{1}{7} &
% 		\MyFigEuler{1}{8} &
% 		\MyFigEuler{1}{9} \\
% 		$t=5/8$ & $t=3/4$ & $t=7/8$ & $t=1$
% 	}
% 	\caption{% 
% 		Display of $\g_{0,k}$ showing the evolution of the fluid particles from $x$ to $g^\star(x)=\min(2x,2-2x)$ for $x \in [0,1]$. The corresponding time is $t=\frac{k-1}{K-1} \in [0,1]$.
% 	}
%    \label{fig:mm1}
% \end{figure}

% \begin{figure}[h!]
% 	\centering
% 		\TabFive{
% 		\MyFigEuler{2}{1} &
% 		\MyFigEuler{2}{2} &
% 		\MyFigEuler{2}{3} &
% 		\MyFigEuler{2}{4} &
% 		\MyFigEuler{2}{5} \\
% 		$t=0$ & $t=1/8$ & $t=1/4$ & $t=3/8$ & $t=1/2$
% 	}
% 	\TabFour{
% 		\MyFigEuler{2}{6} &
% 		\MyFigEuler{2}{7} &
% 		\MyFigEuler{2}{8} &
% 		\MyFigEuler{2}{9} \\
% 		$t=5/8$ & $t=3/4$ & $t=7/8$ & $t=1$
% 	}
% 	\caption{% 
% 		Same as Figure~\ref{fig:mm1} for the map $g^\star(x)=(x+1/2) \text{ mod } 1$ for $x\in [0,1]$. 
% 	}
%    \label{fig:mm2}
% \end{figure}

\begin{figure}[h!]
	\centering
		\TabFive{
		\MyFigEuler{3}{1} &
		\MyFigEuler{3}{2} &
		\MyFigEuler{3}{3} &
		\MyFigEuler{3}{4} &
		\MyFigEuler{3}{5} \\
		$t=0$ & $t=1/8$ & $t=1/4$ & $t=3/8$ & $t=1/2$
	}
	\TabFour{
		\MyFigEuler{3}{6} &
		\MyFigEuler{3}{7} &
		\MyFigEuler{3}{8} &
		\MyFigEuler{3}{9} \\
		$t=5/8$ & $t=3/4$ & $t=7/8$ & $t=1$
	}
	\caption{%  
Non Periodic Case : Gray-map  value of $P_{t_k}$ (see (\ref{ff1}))  for different times $t_k$. Horizontal axis is $x_0$ and vertical axis $x_k$. 
	}
   \label{fig:mm3}
\end{figure}

\begin{figure}
 \centering
 \includegraphics[scale=0.4]{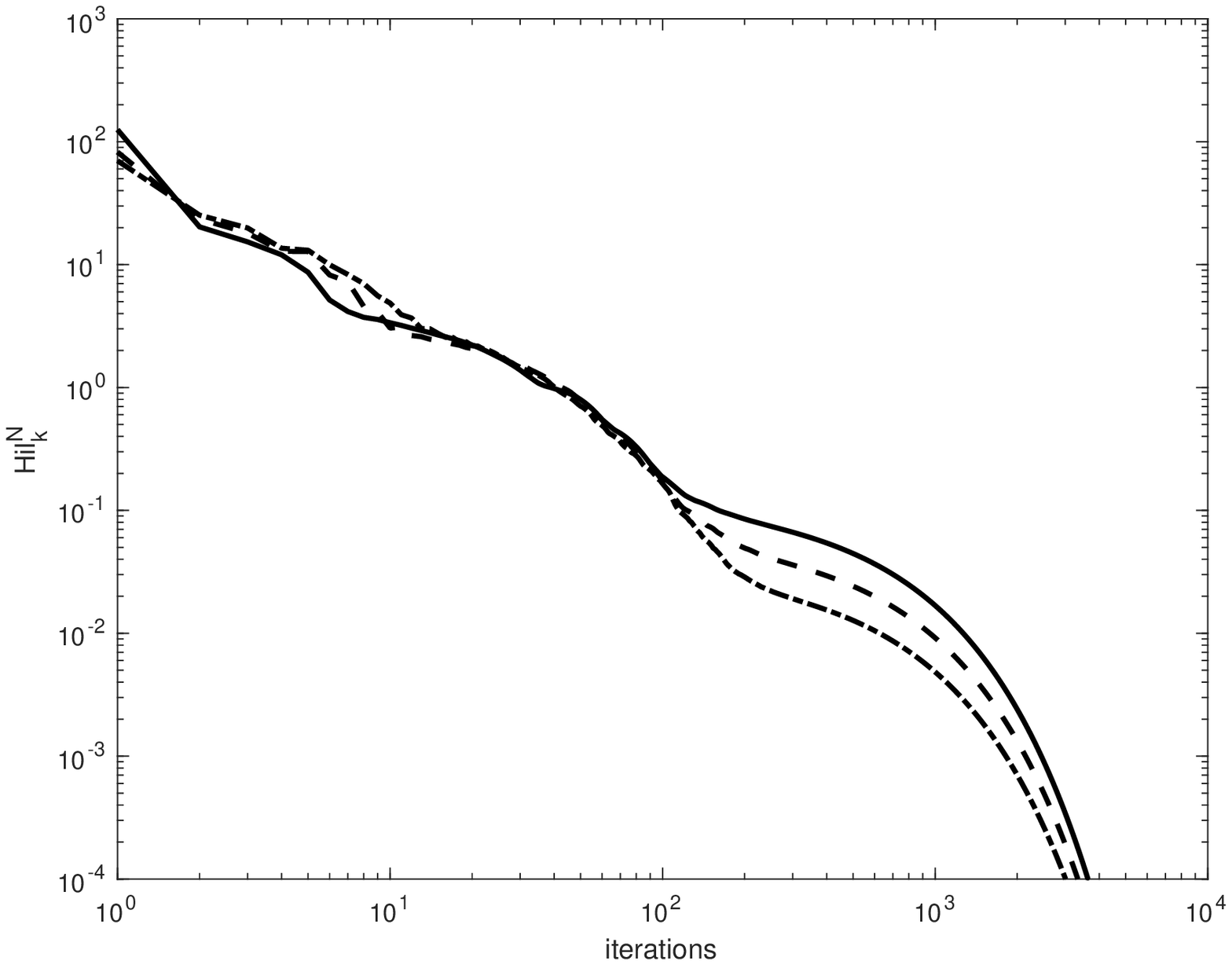}
 \caption{Distance (log scale) between two consecutive iterations of Sinkhorn by using the Hilbert projective metric for $N=8$ (solid line), $N=16$ (dashed line) and $N=32$ (dotted-dashed line)}
 \label{fig:it-hil}
\end{figure}

We also present in figure \ref{fig:mm4} the numerical solution to the same problem 
on the flat torus $\R/\Z$. This can be simply achieved by using the 
periodization of the Euclidean distance $\dist(x,y):=\inf_{k\in \Z} \vert x-y + k\vert\, \; \forall (x,y)\in [0,1]$.
The kernel $K$ is changed accordingly.   The periodization induces a topological change in the 
solution which is classical in periodic  optimal transport problems. In figure \ref{fig:mm5} we present a numerical solution (on the flat torus) for a discontinuous final configuration. 
The computations in figures \ref{fig:mm3}, \ref{fig:mm4} and \ref{fig:mm5} are performed with a uniform discretization of $[0,1]$ with $M=200$ points, $\epsilon=10^{-3}$ and $N=16$.
\begin{figure}[h!]
	\centering
		\TabFive{
		\MyFigEuler{4}{1} &
		\MyFigEuler{4}{2} &
		\MyFigEuler{4}{3} &
		\MyFigEuler{4}{4} &
		\MyFigEuler{4}{5} \\
		$t=0$ & $t=1/8$ & $t=1/4$ & $t=3/8$ & $t=1/2$
	}
	\TabFour{
		\MyFigEuler{4}{6} &
		\MyFigEuler{4}{7} &
		\MyFigEuler{4}{8} &
		\MyFigEuler{4}{9} \\
		$t=5/8$ & $t=3/4$ & $t=7/8$ & $t=1$
	}
	\caption{%  
Periodic Case : Gray-map  value of $P_{t_k}$ (see (\ref{ff1}))  for different times $t_k$. Horizontal axis is $x_0$ and vertical axis $x_k$. 
	}
   \label{fig:mm4}
\end{figure}
\begin{figure}[h!]
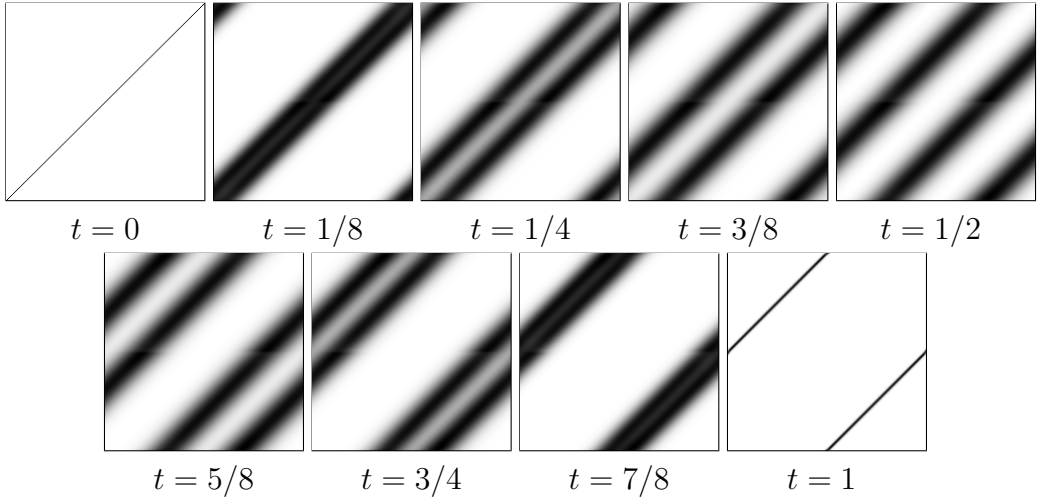

	\centering
		\TabFive{
		\MyFigEuler{5}{1} &
		\MyFigEuler{5}{2} &
		\MyFigEuler{5}{3} &
		\MyFigEuler{5}{4} &
		\MyFigEuler{5}{5} \\
		$t=0$ & $t=1/8$ & $t=1/4$ & $t=3/8$ & $t=1/2$
	}
	\TabFour{
		\MyFigEuler{5}{6} &
		\MyFigEuler{5}{7} &
		\MyFigEuler{5}{8} &
		\MyFigEuler{5}{9} \\
		$t=5/8$ & $t=3/4$ & $t=7/8$ & $t=1$
	}
	\caption{%  
Periodic Case ($X_T$ discontinuous) : Gray-map  value of $P_{t_k}$ (see (\ref{ff1}))  for different times $t_k$. Horizontal axis is $x_0$ and vertical axis $x_k$. 
	}
   \label{fig:mm5}
\end{figure}

\subsection{Two-dimensional experiments: the Beltrami flow}

Consider the unit square $\D=[0,1]^2$ and the Beltrami flow obtained from the following time-independent velocity and pressure fields:
\begin{align*} 
 u(\tx_1,\tx_2)&=(-\cos(\pi \tx_1)\sin(\pi \tx_2),\sin(\pi \tx_1)\cos(\pi \tx_2)),\\
 p(\tx_1,\tx_2)&=\dfrac{1}{2}(\sin(\pi \tx_1)^2+\sin(\pi \tx_2)^2).
 \end{align*}
 where $(\tx_1,\tx_2)$ denotes the 2D  Cartesian coordinates
One can verify that they solve the steady Euler equations. \\
 
 It is possible to integrate the ODE $ \partial_t {X}(t,x) = u(X(t,x)) , \,\, \forall x \in \D $ and construct 
 for any final time $T$  classical  solution in $\Sdif$ to a (\ref{arnold}).  For the same final configuration $X_T = X(T,.)$  Brenier 
 established in  \cite[Theorem 5.1]{brenier1989least})  the consistency with the generalized solution of (\ref{brerelax}) provided
 \[
  \sup_{(t,x) \in  [ 0,T]\times \D}   \nabla_x^2p(t,x)  <  \frac{\pi^2}{T^2} \, Id  
 \]
 in the sense of positive definite matrices.   For the Beltrami flow,  
 the maximum eigenvalue of the Hessian of the pressure is  $\pi^2$ which suggests  $T=1$  is  a critical time and that we may 
 expect the optimal $\GIF$ to depart from the $\Sdif$ solution for $T>1$ and exhibit generalized behavior such as splitting/mixing/crossing 
  of "generalized particles".  \\
  
In order to track the Lagrangian behavior of the generalized solution, we extend the 1D representation technique as follows.
We split the domain $\D$ at initial time into three colored sub-domains~:
\begin{align*}
RED:=[0, \frac{1}{3}]\times [0,1], \; GREEN:= (\frac{1}{3}, \frac{2}{3} ]\times [0,1], \; BLUE:=(\frac{2}{3}, 1]\times [0,1]
% RED&\ddd \{(\tx_1,\tx_2)\in\bigg[0,\frac{1}{3}\bigg]\times\big[0,1\big]\},\\
 %GREEN&\ddd \{ (\tx_1,\tx_2)\in\bigg(\frac{1}{3},\frac{2}{3}\bigg]\times\big[0,1\big]\},\\
 %BLUE&\ddd \{ (\tx_1,\tx_2)\in\bigg(\frac{2}{3},1\bigg]\times\big[0,1\big]\},
 \end{align*}
 Then we use again the definition (\ref{ff1}) and plot for each time $t_k$ and for all $x_k \in \D$  the 2D fields  
 \begin{align}
\label{RGB}  
P_{t_k,R}(x_k) =  \int_{RED} P_{t_k}(x_0,x_k) \, \mbox{d} x_0   \\
P_{t_k,G}(x_k)  = \int_{GREEN} P_{t_k}(x_0,x_k) \, \mbox{d} x_0   \\
P_{t_k,B}(x_k)  = \int_{BLUE} P_{t_k}(x_0,x_k) \, \mbox{d} x_0   
 \end{align}
in the corresponding color  with an {\em opacity} depending the actual value of the field.  
Each of these fields represent the amount of mass which has traveled from the initial RED/BLUE/GREEN region at time $t_k$. 
We also plot the sum of the three fields. The value  at any time is the Lebesgue measure but because we use different 
colors with different opacities, it gives an idea of the mixing.  \\

 In figures \ref{fig:beltrami09}, \ref{fig:beltrami13} and \ref{fig:beltramiPi}, we plot  different final  times $T$~:
 the classical Lagrangian solution with no mixing in the first column, $P_R+P_G+P_B$ in the second column and 
 $P_R$/$P_G$/$P_B$ in the remaining three columns.  \\
 
 For $T=0.9$  (figure \ref{fig:beltrami09}) the classical and $\GIF$ solution agree. One should keep in mind that we solve 
 an Entropic regularization of the problem which should be accounted for some of the mass spreading. \\
 
 For $T=1.3$  (figure \ref{fig:beltrami13}) we are past the critical time and the $\GIF$ exhibit a different behavior with a different 
 pattern like a  clockwise rotation in the middle of the domain. It is cheaper in terms of kinetic energy to send the mass across 
 rather that doing the full counterclockwise rotation.
 
 For $T = \pi$  (figure \ref{fig:beltramiPi}) ,  the $\GIF$ is again different but seems to produce less mixing. \\
  
Notice that our generalized Beltrami solutions are consistent with the solution in  \cite{memi} which are computed using a 
non convex Lagrangian formulation.  \\
All the computations are performed with a uniform discretization of $[0,1]^2$ with $M=64\times64$ points, $\epsilon=10^{-4}$ and $N=16$.
%{\bf LUCA : oui plus de detail sur les discretizations accessibles ...} 
 All these simulation take approximately a CPU time of 3 hours, this means that the code must be parallelized in order to become competitive. \\

\newcommand{\MyEulerBeltrami}[1]{\imgbox{\includegraphics[width=.21\linewidth]{#1}}}
\newcommand{\MyEulerBeltramiODE}[2]{\imgbox{\includegraphics[width=.21\linewidth]{EulerBeltramiT#1timestep#2}}}

\begin{figure}[h!]
	\centering
	\TabFive{
	\MyEulerBeltramiODE{9}{1}&
	\MyEulerBeltrami{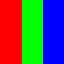}&
	\MyEulerBeltrami{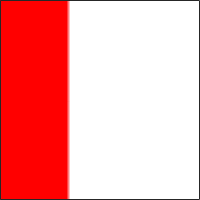}&
	\MyEulerBeltrami{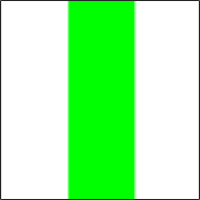}&
	\MyEulerBeltrami{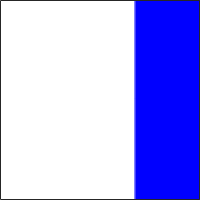}\\
	 & &$t=0$& &  \\
	 \MyEulerBeltramiODE{9}{3}&
	\MyEulerBeltrami{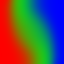}&
	\MyEulerBeltrami{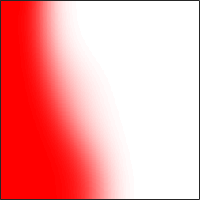}&
	\MyEulerBeltrami{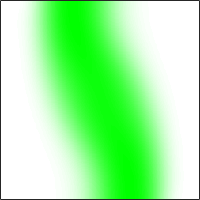}&
	\MyEulerBeltrami{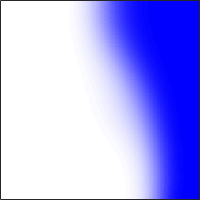}\\
	&&$t=\frac{T}{4}$&& \\
	 \MyEulerBeltramiODE{9}{5}&
	\MyEulerBeltrami{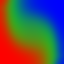}&
	\MyEulerBeltrami{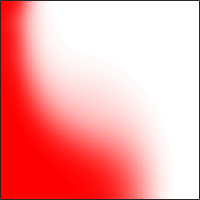}&
	\MyEulerBeltrami{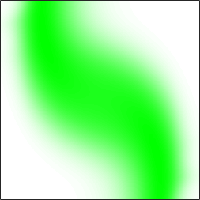}&
	\MyEulerBeltrami{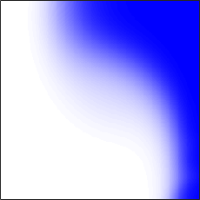}\\
	&&$t=\frac{T}{2}$&& \\
	\MyEulerBeltramiODE{9}{7}&
	\MyEulerBeltrami{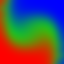}&
	\MyEulerBeltrami{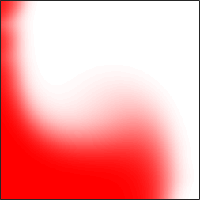}&
	\MyEulerBeltrami{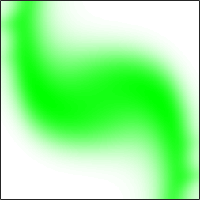}&
	\MyEulerBeltrami{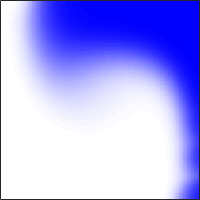}\\
	&&$t=\frac{3T}{4}$&& \\
	\MyEulerBeltramiODE{9}{9}&
	\MyEulerBeltrami{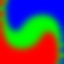}&
	\MyEulerBeltrami{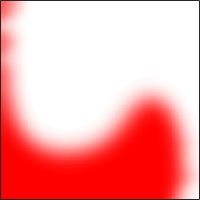}&
	\MyEulerBeltrami{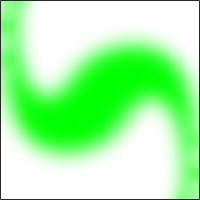}&
	\MyEulerBeltrami{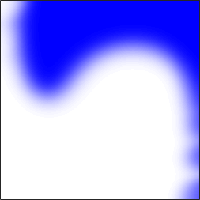}\\
	&&$t=T$&& \\

	}

	\caption{% 
Final time : $T = 0.9$. Columns : Classical Color tracking of the  Lagrangian solution with no mixing in the first column, $P_R+P_G+P_B$ in the second column and 
 $P_R$/$P_G$/$P_B$ in the remaining three columns  (see (\ref{RGB}) for definitions). Rows : Time evolution. The final Lagrangian configuration 
at the bottom left is the final datum $X_T$  in $\pi_{0,T} =(\id, X_T)_\# \Leb$. }
   \label{fig:beltrami09}
\end{figure}

\newcommand{\MyEulerBeltramiODEbis}[2]{\imgbox{\includegraphics[width=.21\linewidth]{EulerBeltramiT#1timeStep#2}}}

\begin{figure}[h!]
	\centering
	\TabFive{
	\MyEulerBeltramiODEbis{13}{1ODE}&
	\MyEulerBeltrami{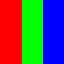}&
	\MyEulerBeltrami{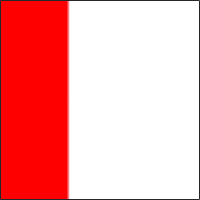}&
	\MyEulerBeltrami{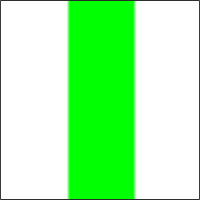}&
	\MyEulerBeltrami{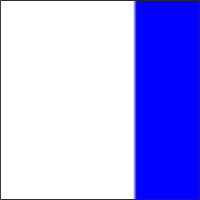}\\
	 & &$t=0$& &  \\
	 \MyEulerBeltramiODEbis{13}{3ODE}&
	\MyEulerBeltrami{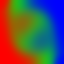}&
	\MyEulerBeltrami{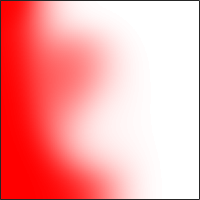}&
	\MyEulerBeltrami{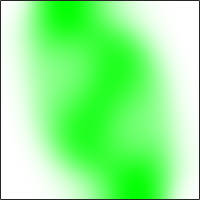}&
	\MyEulerBeltrami{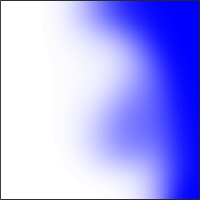}\\
	&&$t=\frac{T}{4}$&& \\
	 \MyEulerBeltramiODEbis{13}{5ODE}&
	\MyEulerBeltrami{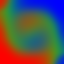}&
	\MyEulerBeltrami{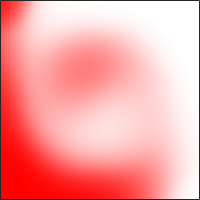}&
	\MyEulerBeltrami{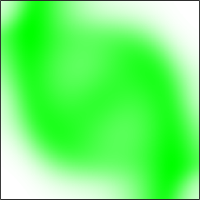}&
	\MyEulerBeltrami{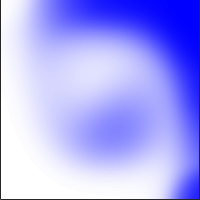}\\
	&&$t=\frac{T}{2}$&& \\
	\MyEulerBeltramiODEbis{13}{7ODE}&
	\MyEulerBeltrami{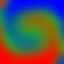}&
	\MyEulerBeltrami{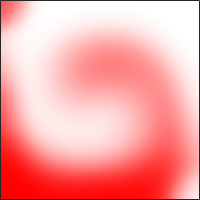}&
	\MyEulerBeltrami{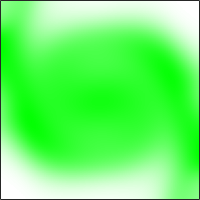}&
	\MyEulerBeltrami{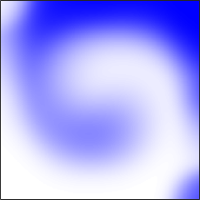}\\
	&&$t=\frac{3T}{4}$&& \\
	\MyEulerBeltramiODEbis{13}{9ODE}&
	\MyEulerBeltrami{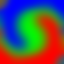}&
	\MyEulerBeltrami{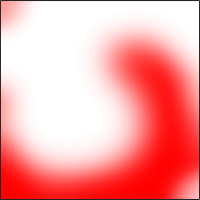}&
	\MyEulerBeltrami{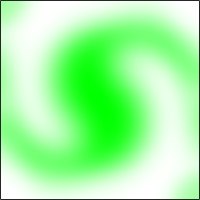}&
	\MyEulerBeltrami{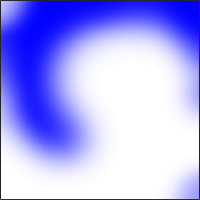}\\
	&&$t=T$&& \\

	}

	\caption{% 
	Final time : $T = 1.3$. Columns : Classical Color tracking of the  Lagrangian solution with no mixing in the first column, $P_R+P_G+P_B$ in the second column and 
 $P_R$/$P_G$/$P_B$ in the remaining three columns  (see (\ref{RGB}) for definitions). Rows : Time evolution. The final Lagrangian configuration 
at the bottom left is the final datum $X_T$  in $\pi_{0,T} =(\id, X_T)_\# \Leb$. } 
   \label{fig:beltrami13}
\end{figure}

\begin{figure}[h!]
	\centering
		\TabFive{
	\MyEulerBeltramiODEbis{pi}{1ODE}&
	\MyEulerBeltrami{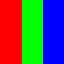}&
	\MyEulerBeltrami{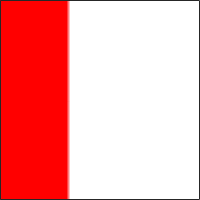}&
	\MyEulerBeltrami{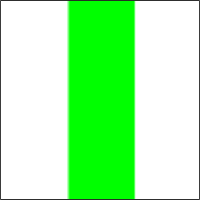}&
	\MyEulerBeltrami{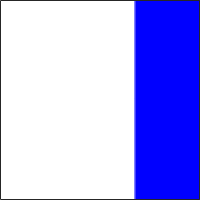}\\
	 & &$t=0$& &  \\
	 \MyEulerBeltramiODEbis{pi}{3ODE}&
	\MyEulerBeltrami{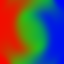}&
	\MyEulerBeltrami{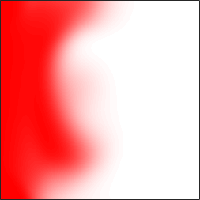}&
	\MyEulerBeltrami{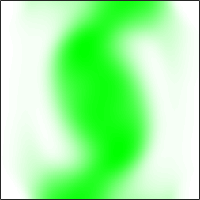}&
	\MyEulerBeltrami{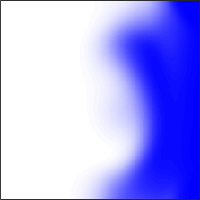}\\
	&&$t=\frac{T}{4}$&& \\
	 \MyEulerBeltramiODEbis{pi}{5ODE}&
	\MyEulerBeltrami{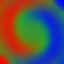}&
	\MyEulerBeltrami{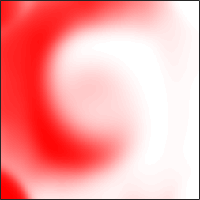}&
	\MyEulerBeltrami{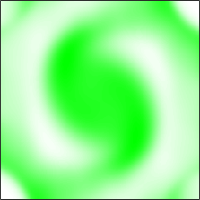}&
	\MyEulerBeltrami{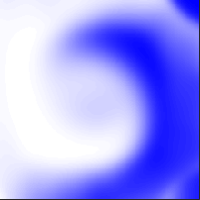}\\
	&&$t=\frac{T}{2}$&& \\
	\MyEulerBeltramiODEbis{pi}{7ODE}&
	\MyEulerBeltrami{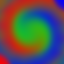}&
	\MyEulerBeltrami{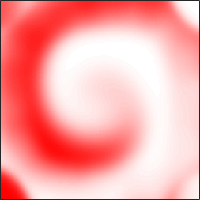}&
	\MyEulerBeltrami{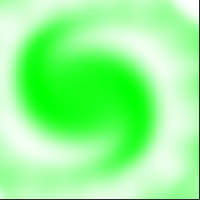}&
	\MyEulerBeltrami{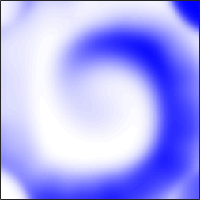}\\
	&&$t=\frac{3T}{4}$&& \\
	\MyEulerBeltramiODEbis{pi}{9ODE}&
	\MyEulerBeltrami{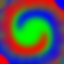}&
	\MyEulerBeltrami{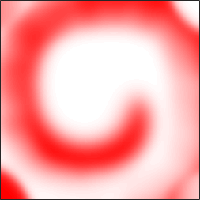}&
	\MyEulerBeltrami{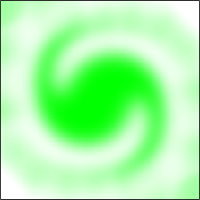}&
	\MyEulerBeltrami{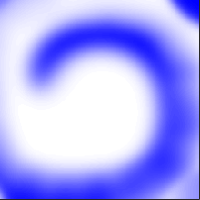}\\
	&&$t=T$&& \\	
	
	}

	\caption{% 
	Final time : $T = \pi$. Columns : Classical Color tracking of the  Lagrangian solution with no mixing in the first column, $P_R+P_G+P_B$ in the second column and 
 $P_R$/$P_G$/$P_B$ in the remaining three columns  (see (\ref{RGB}) for definitions). Rows : Time evolution. The final Lagrangian configuration 
at the bottom left is the final datum $X_T$  in $\pi_{0,T} =(\id, X_T)_\# \Leb$. }
   \label{fig:beltramiPi}
\end{figure}

{\bf{Acknowledgements:}} It is our pleasure to thank Christian L\'eonard and Yann Brenier for many fruitful discussions, we are also grateful to Christian L\'eonard  for sharing a preliminary version of \cite{leonardBredinger} with us. The authors are grateful to the Agence Nationale de La Recherche through the projects ISOTACE and MAGA.  
%\end{acknowledgements}%L.N.'s work has been funded by the European Research Council (ERC) under the European Union's Horizon 2020 research and innovation programme (grant agreement MDFT No 725528). 
 \bibliographystyle{plain} 

\bibliography{bibli}

\end{document}